\documentclass{amsart}
\usepackage{amssymb}

\usepackage{graphicx}
\usepackage{float}

\newtheorem{theorem}{Theorem}
\newtheorem{corollary}[theorem]{Corollary}

\theoremstyle{remark}

\numberwithin{equation}{section}

\newcommand\mL{L\kern-0.08cm\char39}

\newcommand{\rad}{\mathop{\rm rad}}
\newcommand{\diam}{\mathop{\rm diam}}

\begin{document}

\begin{large}

\title[]{On cycles in graphs with specified radius and diameter}

\author[P. Hrn\v ciar]{Pavel Hrn\v ciar}
\address{Department of Mathematics, Faculty of Natural Sciences,
          Matej Bel University, Tajovsk\'eho 40, 974 01 Bansk\'a Bystrica,
          Slovakia}
\email{Pavel.Hrnciar@umb.sk}

\thanks{The author was supported by the Slovak Grant Agency under
        the grant number VEGA 1/1085/11.}

\subjclass[2000]{05C12}

\keywords{Radius, diameter, cycle, circumference, path, geodesic cycle}

\begin{abstract}
Let $G$ be a graph of radius $r$ and diameter $d$ with $d\leq 2r-2$. We show that $G$ contains a cycle of  length at least $4r-2d$, i.e. for its circumference it holds $c(G)\geq 4r-2d$. Moreover, for all positive integers $r$ and $d$ with $r\leq d\leq 2r-2$ there exists a graph of radius $r$ and diameter $d$ with circumference $4r-2d$.
\end{abstract}

\maketitle



For a connected graph $G$, the \emph{distance} $d_G(u,v)$ or briefly $d(u,v)$ between a pair of vertices $u$ and $v$ is the length of a shortest path joining them. The distance between a vertex $u\in V(G)$ and a subgraph $H$ of $G$ will be denoted by $d(u,H)$, i.e. $d(u,H)=\min\{d(u,v); v\in V(H)\}$.
The \emph{eccentricity} $e_G(u)$ (briefly $e(u)$) of a vertex $u$ of $G$ is the distance of $u$ to a vertex farthest from $u$ in $G$, i.e. $e_G(u)=\max\{d_G(u,v); v\in V(G)\}$.
The \emph{radius} $\rad G$ of $G$ is the minimum eccentricity among the vertices of $G$ while the \emph{diameter} $\diam G$ of $G$ is the maximum eccentricity.
The \emph{circumference} of a graph $G$, denoted $c(G)$, is the length of any longest cycle in $G$.

A path $P$ (a cycle $C$) in $G$ is called \emph{geodesic} if for any two vertices of $P$ (of $C$) their distance in $P$ (in $C$) equals their distance in $G$.
A nontrivial connected graph with no cut-vertices is called a \emph{nonseparable graph}. A \emph{block} of a graph $G$ is a maximal nonseparable subgraph of $G$.

A connected unicyclic graph $G$ with the cycle $C$ is called a \emph{sun-graph} (see \cite{Hav_Hrn_Mon_2}) if $\deg_G(u)\leq 3$ for $u\in V(C)$ and $\deg_G(u)\leq 2$ for $u\in V(G)\setminus V(C)$. A $u-v$ path $P$ in a sun-graph $G$ is called a \emph{ray} if $V(P)\cap V(C)=\{u\}$ and $\deg_G(v)=1$. A sun-graph with the cycle $C_m$ of length $m$ and with $m$ rays of length $k$ will be denoted by $S_{m,k}$.

\vskip 3mm
 
In what follows we answer a question that was posed several decades ago in \cite{Ost}:
\newline
"How large a cycle must there be in a graph of radius $m$ and diameter $n$? This question is also open. For radius 3 and diameter 4, the graph must have a cycle of length at least 4, which can be verified by brute force techniques \dots . The situation in general is unclear." 

Our main result is the following theorem.

\begin{theorem}
Let $G$ be a graph of radius $r$ and diameter $d$ with $d\leq 2r-2$. Then $c(G)\geq 4r-2d$.
\end{theorem}

\begin{proof}
Since $d\leq 2r-2$, $G$ is not a tree. Let $C$ be a cycle of $G$ and $B$ be a block of $G$ containing $C$. Suppose, contrary to our claim, that $c(G)<4r-2d$. Since $B$ is a nonseparable subgraph of $G$, every two vertices of $B$ lie on a common cycle (see \cite[Theorem 1.6]{Buc-Har-1}). We get $\diam B\leq 2 r-d-1\leq r-1$ (and so $B\neq G$).

Let $u$ be a vertex such that $d(u,B)=\max \{d(v,B); v\in V(G)\}$ and let $u_B\in V(B)$ be a vertex with $d(u,u_B)=d(u,B)$. Evidently, $u_B$ is a cut-vertex of $G$. 

Let $G_1$ be a component of $G-u_B$ containing the vertex $u$. Put $d(u,u_B)=a$. We distinguish two cases.

\vskip 1.5mm
(1) $a\leq r-1$

Let $v$ be a vertex of $G$. If $v\in V(G_1)$ then $d(v,u_B)\leq a\leq r-1$. If $v\in V(B)$ then $d(v,u_B)\leq\diam B\leq r-1$. Let, finally, $v\in V(G)\setminus (V(G_1)\cup V(B))$ and $v_B\in V(B)$ be a vertex such that $d(v,v_B)=d(v,B)$. Evidently, $v_B$ is a cut-vertex of $G$ and $d_G(u_B,v_B)=d_B(u_B,v_B)$. Denote $d(v,v_B)=b$ and $d(u_B,v_B)=c$. Suppose first that $b+c\geq r$. We have $c\leq\diam B\leq 2r-d-1$ and $b\leq a$. Then $d(u,v)=d(u,u_B)+d(u_B,v_B)+ d(v_B,v)=a+c+b\geq 2b+c\geq 2(r-c)+c\geq 2r-(2r-d-1)=d+1$. Since $\diam G=d$, we get $b+c\leq r-1$ and so $d(u_B,v)\leq r-1$. Finally, we have $e(u_B)\leq r-1$, a contradicton.
\vskip 1.5mm
(2) $a\geq r$

Let $u_1$ be a vertex of a geodesic $u-u_B$ path $P^1$ with $d(u,u_1)=r-1$. If $w$ is a vertex from $V(G)\setminus V(G_1)$ then $u_B$ is on a geodesic $w-u_1$ path and we get $d(w,u_1)\leq r-1$ (since $d(w,u)\leq 2r-2$). Since $e(u_1)\geq r$ (otherwise we have a contradiction), there is a vertex $v\in V(G_1)$ such that $d_{G_1}(v,u_1)=d_G(v,u_1)\geq r$. Let $P^2$ be a geodesic $v-u_B$ path and let $v_1$ be the first vertex of $P^2$ which is on $P^1$. Since $d(v,u_B)\leq d(u,u_B)$, we get $d(u_B,v_1)<d(u_B,u_1)$. Let $P^3$ be a geodesic $v-u$ path and let $v_2$ be the first of its vertices which is on $P^1$. It is obvious (since $d(v,u)\leq 2r-2$) that $d(u_B,v_2)>d(u_B,u_1)$. Evidently, there is a cycle $C'$ such that $\{v_1,v_2\}\subseteq V(C')$.

Let $G_2$ be a subgraph of $G$ induced by the set $V(G_1)\cup\{u_B\}$. Let $w\in V(G)\setminus V(G_2)$ be such a vertex that $d(w,u_B)=\max\{d(x,u_B); x\in V(G)\setminus V(G_2)\}$ and $P$ be a geodesic $w-u_B$ path. Consider a graph $G'$ for which $V(G')=V(G_2)\cup V(P)$ and $E(G')=E(G_2)\cup E(P)$. It is obvious that $|V(G')|<|V(G)|$ and if there is a vertex $z\in V(G')$ with $e_{G'}(z)\leq r-1$ then $e_G(z)\leq r-1$, too.

\vskip 2mm

We can repeat the previous considerations with the graph $G'$ and its block $B'$ containing the cycle $C'$. It is clear now that after a finite number of the described steps we get a contradiction.
\end{proof}

\begin{corollary}
If $G$ is a graph with $\rad G=r$ and $\diam G\leq 2r-2$, then $G$ contains a cycle of length at least 4, i.e. $c(G)\geq 4$.
\end{corollary}

\begin{corollary}
If $c(G)=3$ and $\rad G=r$, then $\diam G\in\{2r-1, \,2r\}$.
\end{corollary}

\vskip 0.5cm

For all positive integers $r$ and $d$ satisfying $r\leq d\leq 2r-2$ there exists an infinite number of graphs of radius $r$, diameter $d$ and circumference $4r-2d$. One of these graphs is $C_{2r}$ for $d=r$. If $d>r$, one of these graphs is   $S_{4r-2d,d-r}$, i.e. a sun-graph with the cycle $C_{4r-2d}$ and with $4r-2d$ rays of length $d-r$ (see Figure 1a for $r=3$, $d=4$ and Figure 1b for $r=5$, $d=7$). 
\vskip 2mm
\begin{figure}[H]
\centering
\includegraphics[trim = 5cm 20cm 5cm 4cm, clip]{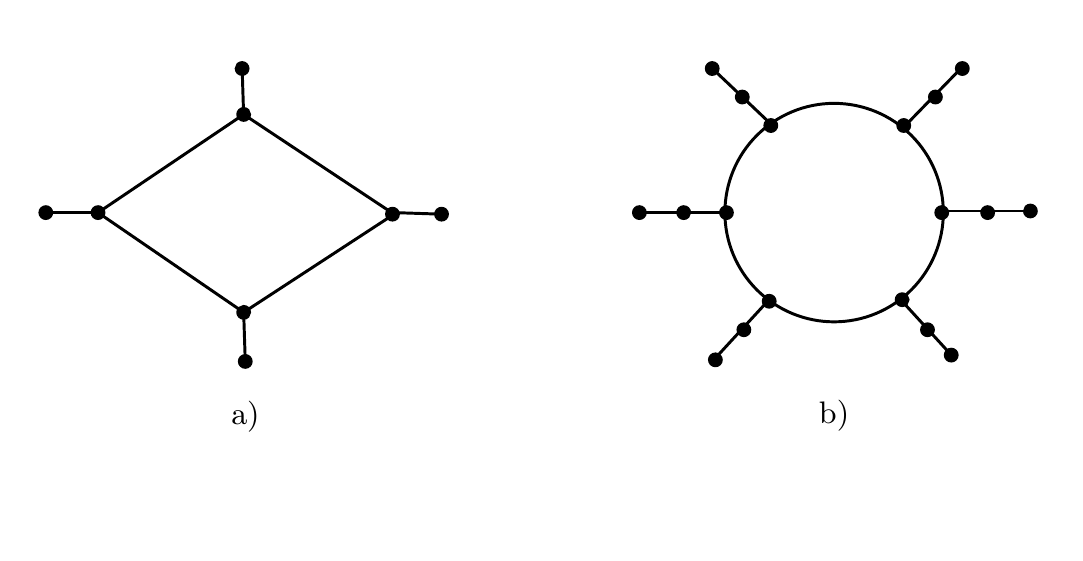}
\caption{}
\end{figure}
\vskip 2mm

\noindent
Now it is a simple matter to find infinite classes of graphs with mentioned properties (see Figure 2 for an inspiration).

\vskip 2mm
\begin{figure}[H]
\centering
\includegraphics[trim = 5cm 20cm 5cm 4cm, clip]{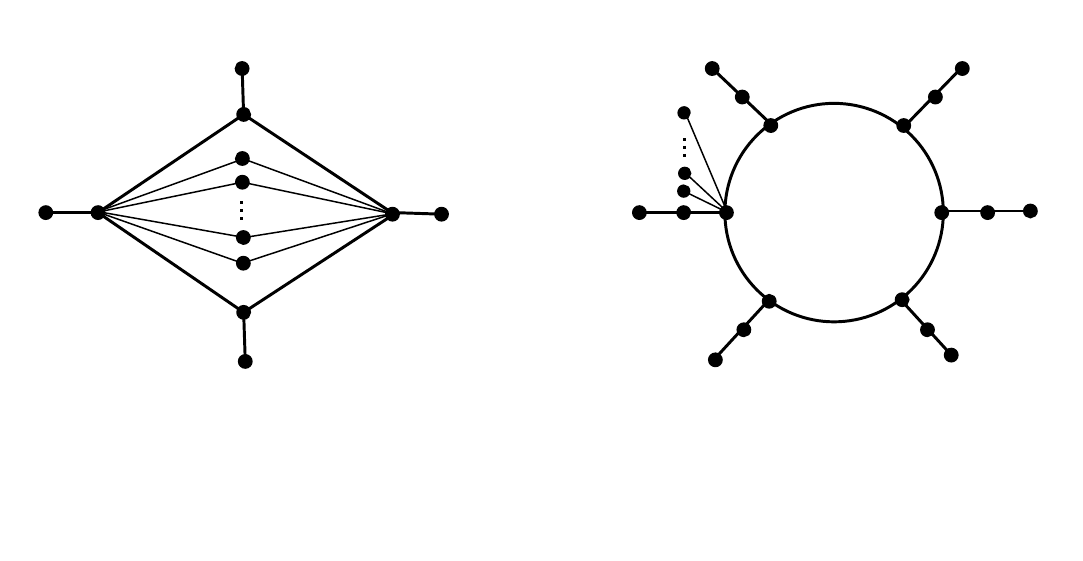}
\caption{}
\end{figure}
\vskip 2mm

It is known that if a graph $G$ with radius $r$ and diameter $d\leq 2r-2$ has at most $3r-2$ vertices, then it holds $c(G)\geq 2r$. This fact is a consequence of the following theorem (see \cite{Hav_Hrn_Mon_2}).

\begin{theorem}\label{T2:C2r_C2r+1}\cite{Hav_Hrn_Mon_2} 
Let $G$ be a graph with $\rad G=r$, $\diam G\leq 2r-2$, on at most $3r-2$ vertices. Then $G$ contains a geodesic cycle of length $2r$ or $2r+1$.
\end{theorem}

Using Theorem \ref{T2:C2r_C2r+1} it is easy to find all nonisomorphic graphs of minimal order and specified radius and diameter (see \cite{Ost},\cite{Hav_Hrn_Mon_2}).

Let $G$ be a sun-graph with the cycle $C_{2r-1}$ ($r\geq 3$), with $r$ rays of length 1 and such that exactly two of its end-vertices have distance 3 (see Figure 3 for $r=5$). It is easy to see that $\rad G=r$, $|V(G)|=3r-1$ and $c(G)=2r-1$. 

\vskip 2mm
\begin{figure}[H]
\centering
\includegraphics[trim = 8cm 20.5cm 8cm 4.5cm, clip]{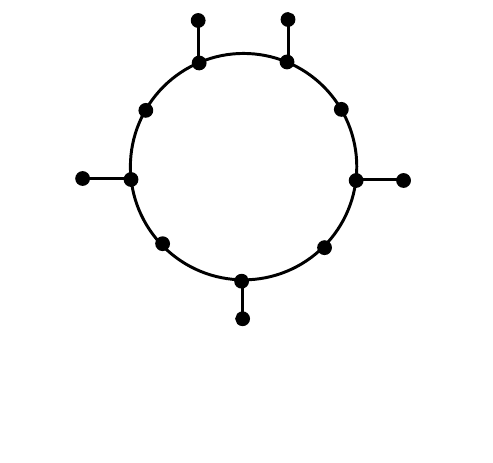}
\caption{}
\end{figure}
\vskip 2mm

\noindent
We can conclude that the bound $3r-2$ in Theorem 4 is the best possible (for $r\geq 3$).

\end{large}


\begin{thebibliography}{99}
\bibitem{Buc-Har-1}
   F. Buckley and F. Harary,
   \emph{Distance in Graphs}, 
   Addison-Wesley Publishing Company, Redwood City, CA, 1990.

   
 
   
\bibitem{Hav_Hrn_Mon_2}
   A. Haviar, P. Hrn\v ciar and G. Monoszov\'a,  
   \emph{Eccentric sequences and cycles in graphs},
   Acta Univ. M. Belii, Ser. Math.~no 11~(2004), 7--25.

   

    
\bibitem{Ost}
   P. A. Ostrand,
   \emph{Graphs with specified radius and diameter},
   Discrete Math.~\textbf{4} (1973), 71--75.  


\end{thebibliography}
\end{document}